\documentclass[10pt, leqno]{article}
\newenvironment{proof}{\noindent {\bf Proof:}}{$\Box$ \vspace{2 ex}}

\usepackage{amsmath, amssymb, amscd}
\usepackage{latexsym, epsfig, color, url, relsize}
\usepackage[all]{xy}
\usepackage{tikz, verbatim, xcolor}

\newcommand{\C}{\mathbb{C}}

\newcommand{\chr}{{\rm{char}}}
\newcommand{\Z}{\mathbb{Z}}

\newcommand{\Q}{\mathbb{Q}}
\newcommand{\R}{\mathbb{R}}
\newcommand{\F}{\mathbb{F}}
\newcommand{\A}{\mathbb{A}}
\newcommand{\cO}{\mathcal{O}}
\newcommand{\co}{\mathcal{O}}

\newcommand{\beq}{\begin{equation}}
\newcommand{\eeq}{\end{equation}}

\newcommand{\calF}{\mathcal{F}}

\newcommand{\sillyX}{\mathfrak{X}}

\newcommand\Vol{\operatorname{Vol}}
\newcommand\Disc{\operatorname{Disc}}

\newcommand\GL{\operatorname{GL}}
\newcommand\gl{\operatorname{GL}}

\newcommand\fc{\operatorname{fc}}

\renewcommand\b\bullet
\renewcommand\c\times

\definecolor{dgreen}{RGB}{0, 170, 0}

\newcommand{\odr}{\co_{D1^2}}		
\newcommand{\ods}{\co_{D11}}		
\newcommand{\odi}{\co_{D2}}		
\newcommand{\odg}{\co_{D\rm{ns}}}	
\newcommand{\ocs}{\co_{C\rm{s}}}	
\newcommand{\ocg}{\co_{C\rm{ns}}}	
\newcommand{\ots}{\co_{T11}}		
\newcommand{\oti}{\co_{T2}}		

\newcommand{\sump}{\sideset{}{'}\sum}

\newcommand\ldmod{q}

\newcommand\Sym{\operatorname{Sym}}

\newtheorem{proposition}{Proposition}
\newtheorem{theorem}[proposition]{Theorem}

\newtheorem{conclusion}[proposition]{Conclusion}
\newtheorem{lemma}[proposition]{Lemma}
\newtheorem{remark}[proposition]{Remark}

\title{Levels of distribution for sieve problems in prehomogeneous vector spaces}
     
\author{Takashi Taniguchi and Frank Thorne}
       
\begin{document}

\maketitle

\begin{abstract}
In our companion paper \cite{TT_orbital}, we developed an efficient algebraic method for computing the Fourier transforms
of certain functions defined on prehomogeneous vector spaces over finite fields, and we carried out these computations in a variety of cases.

Here we develop a method, based on Fourier analysis and algebraic geometry, which exploits these Fourier transform formulas
to yield {\itshape level of distribution} results, in the sense of analytic number theory. Such results are of the shape typically
required for a variety of sieve methods. As an example  of such an 
application we prove that there are $\gg \frac{X}{\log X}$ quartic fields whose discriminant is squarefree, bounded above by $X$,
and has at most eight prime factors.
\end{abstract}

In this paper we will develop a general technique sufficient to prove the following:

\begin{theorem}\label{thm:ap3}
There is an absolute constant $C_3 > 0$ such that for each $X > 0$,
there exist $\geq (C_3 + o_X(1)) \frac{X}{\log X}$ cubic fields $K$
whose discriminant is squarefree, bounded above by $X$, and has at most 
$3$ prime factors.
\end{theorem}  

\begin{theorem}\label{thm:ap4}
There is an absolute constant $C_4 > 0$ such that for each $X > 0$,
there exist $\geq (C_4 + o_X(1)) \frac{X}{\log X}$ quartic fields $K$
whose discriminant is squarefree, bounded above by $X$, and has at most 
$8$ prime factors.
\end{theorem}

Theorem \ref{thm:ap3} improves on a result of Belabas and Fouvry \cite{BF} (which in turn improved upon Belabas \cite{bel_sieve}), 
where they obtained the same result with $7$ in place of our $3$, and
our methods are in large part a further development
of their ideas.
In \cite{BF} they remarked that introducing a {\itshape weighted sieve} would lower this $7$ to $4$; the further improvement to $3$ comes from an
improvement in the corresponding {\itshape level of distribution}.
The application to quartic fields is, to our knowledge, new.

Besides the weighted sieve, the main ingredients of our method are unusually strong estimates 
for the relevant 
exponential sums (which we obtained in \cite{TT_orbital}), and a suitably adapted version of the recently established Ekedahl-Bhargava geometric sieve \cite{B_geosieve}.

The method is designed to yield strong {\itshape level of distribution estimates} as inputs to {\itshape sieve methods for prehomogeneous vector spaces}. We consider the following setup:

\begin{enumerate}
\item
A sieve begins with a set of objects $\mathcal{A}$ to be sieved. Here this is defined by a representation $V$
of an algebraic group $G$ which we assume to be {\itshape prehomogeneous}: 
the action of $G(\C)$ on $V(\C)$ has a Zariski-open orbit, defined by the nonvanishing
of a polynomial which we call the {\itshape discriminant}. 
We assume here that $\Disc(gx) = \Disc(x)$ identically 
for all $g \in G(\Z)$, and take as our set $\mathcal{A} = \mathcal{A}(X)$ the set of $G(\Z)$-orbits $x \in V(\Z)$ with
$0 < |\Disc(x)| < X$.

Prehomogeneous vector spaces are the subject of a wealth of parametrization theorems (see, e.g., \cite{WY, HCL1, HCL2, HCL3, HCL4})
and corresponding theoerems concerning the arithmetic objects being parametrized. We refer also to \cite{BH} for a large
number of interesting {\itshape coregular} (not prehomogneous) examples, for which the methods of this paper (and perhaps also
\cite{TT_orbital}) are likely to apply.

\item\label{foo2}
For each prime $p$ we define a notion of an object $x \in \mathcal{A}$ being `bad at $p$'; a typical application of sieve methods is to 
estimate or bound the number of $x \in \mathcal{A}$ which are not bad at any prime $p < Y$,
for some parameter $Y$. For Theorems \ref{thm:ap3} and \ref{thm:ap4} we will take
`bad at $p$' to mean `has discriminant divisible by $p$'. Other defintions of `bad' are also important,
for example the Davenport-Heilbronn nonmaximality condition of \cite{DH, BBP, BST, TT_rc}.

We may work with any definition of `bad' meeting the following technical condition: there is an integer $a \geq 1$ and 
a $G(\Z/p^a\Z)$-invariant 
subset $S_p \subseteq V(\Z/p^a \Z)$, such that $x \in V(\Z)$ is bad if and only if its reduction $\pmod {p^a}$ is in $S_p$.

\item\label{foo}
For each squarefree integer $q$, we require estimates for the number of $x \in \mathcal{A}(X)$ bad at each prime dividing $q$.
These may be obtained via the geometry of numbers (see, among other references, \cite{DH, bel_sieve, BF, B_quintic, BST, EPW}) or using Shintani zeta functions \cite{TT_rc},
and we develop a third (simpler) method here.

Here these estimates will be of the form $C \omega(q) X^{r/d} + E(X, q)$ for a fixed constant $C$ and multiplicative function $\omega$,
and an error term $E(X, q)$ which we want to bound.

\item
In cases where $\mathcal{A}$ depends on a parameter $X$,
a {\itshape level of distribution} is any $\alpha > 0$ for which the sum $\sum_{q \leq X^{\alpha}} |E(X, q)|$ is adequately small.
Typically, and here, a cumulative error $\ll_A X (\log X)^{-A}$ (for each $A > 0$) is more than sufficient -- but see \cite{BST, TT_rc, ST5}
for examples where a power savings is relevant.
\end{enumerate}

We refer to books such as \cite{CM} and \cite{ODC} for examples of sieve methods and applications; many of the methods can be used essentially
as black boxes, for which the level of distribution is the most important input. {\itshape It is the goal of this paper to develop a method
for proving levels of distribution for prehomogeneous vector spaces which are quantitatively as strong as possible.}

\medskip
Typically, an important ingredient is finite Fourier analysis. Let $\Psi_p : V(\Z/p^a \Z) \rightarrow \C$
be the characteristic function of the subset $S_p$ described above, and let $\widehat{\Psi_p}$ be its Fourier transform
(defined by \eqref{eqn:ft_intro} below). We expect upper bounds on $E(X, p)$ to follow from statements to the effect that
that the function $\Psi_p$ is equidistributed, and upper bounds on the $L_1$-norm of $\widehat{\Psi_p}$ constitute a strong quantitatve
statement of this equidistribution. 

The basic heuristic of this paper is that {\itshape $L_1$ norm bounds for Fourier transforms over finite fields should lead to level of distribution statements for arithmetic
objects}. For examples carried out via the geometry of numbers we refer to \cite{BF}, \cite[(80)-(83)]{BST}, and \cite[Proposition 9.2]{FK};
for examples using the Shintani zeta function method we refer to \cite{TT_rc} (espectially Section 3) or the forthcoming work \cite{PT}. 

In the present paper
we develop a simpler version of this idea, not requiring any knowledge of the geometry of the `cusps' or of the analytic
behavior of the zeta function, in the context of a {\itshape lower bound sieve}. For any Schwartz function $\phi$, the Poisson summation formula
takes the form
\begin{align}\label{eq:poisson_error_0}
\sum_{x \in V(\Z)} \Psi_q(x) \phi(xX^{-1/d}) & = 
X^{\dim(V)/d} \sum_{w \in V^*(\Z)} \widehat{\Psi_q}(w) \widehat{\phi}\left(\frac{w X^{1/d}}{q} \right),
\end{align}
and for suitable $\phi$ the left side will be a smoothed undercount of the number of 
$G(\Z)$-orbits $x \in V(\Z)$ with $0 < |\Disc(x)| < X$, satisfying the `local conditions' described by $\Psi_q$. 
The left side of \eqref{eq:poisson_error_0} takes the role of (the counting function of) $\mathcal{A}(X)$, and the function
$\Psi_q$ may be used to detect those $x$ which are `bad at $q$'.

On the right side of \eqref{eq:poisson_error}, the $w = 0$ term the expected main term, and
the rapid decay of $\widehat{\phi}$ will imply that the error term is effectively bounded by a sum of $|\widehat{\Psi_q}(w)|$
over a box of side length $\ll q X^{-1/d}$. 

Therefore we are reduced to bounding sums of 
$|\widehat{\Psi_q}(w)|$ over boxes and over ranges of squarefree $q$. Obviously we require bounds on the individual
$|\widehat{\Psi_q}(w)|$ to proceed. At this point we could incorporate the general bounds of 
Fouvry and Katz \cite{FK}; indeed, our method would quite directly exploit the geometric structure of their results.
However, in the cases of interest much stronger bounds hold for the $|\widehat{\Psi_q}(w)|$ than are proved in \cite{FK}.
We proved this in \cite{TT_orbital} as a special case of exact formulas for $\widehat{\Psi_q}(w)$, for any
of five prehomogeneous vector spaces $V$, any squarefree integer $q$, and any $G(\Z/q\Z)$-invariant function
$\Psi_q$.

Our $L_1$-norm heuristic arises by replacing each $|\widehat{\Psi_q}(w)|$ by its average over $V^*(\Z/q\Z)$. 
In practice there are limits to the equidistribution of $|\widehat{\Psi_q}(w)|$, and
algebraic geometry now takes center stage. For $q = p \neq 2$ prime, the largest values of 
$|\widehat{\Psi_p}(w)|$ will be confined to $w \in \sillyX(\F_p)$ for a {\itshape scheme} $\sillyX$, defined over $\Z$, and of high
codimension. (The same is also true of the Fouvry-Katz bounds.)

We now apply the {\itshape Ekedahl-Bhargava} geometric sieve \cite{B_geosieve}, which essentially bounds the number of 
pairs $(w, p)$ with $w$ in our box and $p$ prime. As we must work with general squarefree integers $q$, and a filtration of schemes
$\sillyX_i$ rather than just one fixed $\sillyX$, we develop a variation of the geometric
sieve adapted to this counting problem.

\medskip
{\itshape Organization of this paper.}
For simplicity we structure our paper around the proof of Theorem \ref{thm:ap4}, even though our more important aim is to present a method
which works in much greater generality. We begin in Section \ref{sec:properties} by introducing the prehomogeneous vector
space $(G, V) = (\GL_2 \times \GL_3, 2\otimes \Sym_2(3))$ and describing
its relevant properties. We take some care to delineate which of its
properties are relevant to the proof, so that the reader can see what is required to adapt our method to other prehomogeneous
vector spaces and to other sieve problems.

In Section \ref{sec:explain_ld} we introduce the {\itshape weighted sieve} of Richert \cite{richert} and Greaves \cite{greaves}, used to conclude Theorems \ref{thm:ap3} and \ref{thm:ap4},
and we also precisely formulate the level of distribution statement which it will require.

In Section \ref{sec:smoothing} we introduce our smoothing method. Our main result is Proposition \ref{prop:gen_ld_simp}, which states that a level
of distribution follows from an essentially combinatorial estimate. The proof is a fairly typical application of Poisson summation, and follows along
lines that should be familiar to experts. The proof is carried out in a general setting, and we state all of our assumptions at the beginning of the section. 

Section \ref{sec:ap3} proves Theorem \ref{thm:ap3}, and may be skipped without loss of continuity. Here the geometry is simple enough that
we may conclude without introducing any algebro-geometric machinery, and the argument may be read as a prototype for the generalities which follow.

In Section \ref{subsec:2sym23-subschemes} we introduce some algebraic geometry to describe the $G(\F_p)$-orbits on $V(\F_p)$ 
in terms of schemes defined over $\Z$. 
This sets up an application of the Ekedahl-Bhargava geometric sieve \cite{B_geosieve}, of which we develop
a variation in Section \ref{sec:geom}.

Section \ref{sec:ld} is the heart of the proof, further developing the geometric sieve to prove
our level of distribution statement. The proof is specalized
to the particular $(G, V)$ and $\Phi_q$ being studied, but the generalization to other cases should be immediate.

Finally,  in Section \ref{sec:ld_implies} we prove Theorem \ref{thm:ap4}. Essentially the
proof is a formal consequence of our level of distribution, although there are a few technicalities pertaining to this particular
$(G, V)$ and its arithmetic interpretation.

\medskip
{\itshape Remark.} The quartic fields produced by Theorem \ref{thm:ap4} will all have Galois group $S_4$, and the theorem still holds
if we specify that $\Disc(K)$ should be positive or negative, or indeed that $K$ has a fixed number of real embeddings. 

There are other methods for producing almost-prime quartic field discriminants,
One, suggested to us by Bhargava, is to specialize $11$ of the $12$ variables in $V$ to particular values and then
apply results on polynomials in one variable; a second is to apply results \cite{CT4} on quartic fields with fixed cubic resolvent. 

In either case we obtain quartic field discriminants with fewer than $8$ prime factors. but it is unclear how to get, say,
$\gg X^{9/10}$ of them, let alone $\gg X (\log X)^{-1}$. Moreover, proving that one obtains {\itshape fundamental} discriminants seems to be
nontrivial with the first method, and the second method intrinsically produces non-fundamental discriminants.

\medskip
{\itshape Notation.} We observe the following conventions in this paper. $x$ will denote a general element of $V(\Z)$ (sometimes only up
to $G(\Z)$-equivalence). $r$ will denote the dimension of $V$ and $d$ will denote the degree of its (homogeneous) `discriminant' polynomial.
$X$ will indicate a discriminant bound, and in Section \ref{sec:explain_ld} we write $X$ and $Y$ in place of the $x$ and $X$ of \cite{ODC}.
$p$ will denote a prime and $q$ a squarefree integer, in contrast to \cite{TT_orbital} where $q$ was used for the cardinality of a finite field.


\section{The `quartic' representation and its essential properties}\label{sec:properties}

For each ring $R$ (commutative, with unit), let $V(R) = R^2\otimes\Sym_2(R^3)$ be
the set of pairs of ternary quadratic forms with coefficients in $R$; when $2$ is not a zero divisor in $R$,
we also regard $V(R)$ as the set of pairs of   $3 \times 3$ symmetric matrices. Let
$G(R) := \GL_2(R) \times \GL_3(R)$;
there is an action of $G$ on $V$, defined by
\[
(g_2, g_3) \cdot (A, B) = (r \cdot g_3 A g_3^T + s \cdot g_3 B g_3^T, t \cdot g_3 A g_3^T + 
u \cdot g_3 B g_3^T),
\]
where $g_2 := \begin{pmatrix} r & s \\ t & u \end{pmatrix}$.

The {\itshape discriminant} is defined (see \cite[p. 1340]{HCL3}) by the equation
\begin{equation}\label{def:disc}
\Disc((A, B)) := \Disc(4 \det(Ax + By)),
\end{equation}
where $4 \det(Ax + By)$ is a binary cubic form in the variables $x$ and $y$, and the second `$\Disc$' above is its
discriminant. 

It was proved by 
Bhargava \cite{HCL3} that the $G(\Z)$-orbits on $V(\Z)$ parametrize quartic rings, in a sense
that we recall precisely in Section \ref{sec:ld_implies}.
(This parametrization, together with a geometry of numbers argument, allowed Bhargava to prove 
\cite{B_quartic} an asymptotic formula for the number of quartic fields
of bounded discriminant.)

\smallskip
We note the following additional facts about this representation. Although we will not attempt to axiomatize our method here,
these are the inputs required to establish a lower bound sieve for $G(\Z)$-orbits on $V(\Z)$. (Arithmetic applications, such as passing from quartic rings to
quartic {\itshape fields} as we do
in Section \ref{sec:ld_implies}, may in some cases require extra steps which will not generalize as readily.)

\begin{enumerate}
\item ({\itshape Homogeneity of the discriminant; definitions of $r$ and $d$.})
By \eqref{def:disc}, $\Disc$ is a homogeneous polynomial of degree $12$ on $V$.
We write $d = 12$ for the degree of this polynomial, and $r := \dim(V) = 12$; these quantities coincide in this and other interesting cases, but not always.

\item ({\itshape Approximation of the fundamental domain.})
Let $\mathcal{F}$ be a fundamental domain for the action of $G(\Z)$ on $V(\R)$, and let 
$\mathcal{F}^1$ be the subset of $x \in \calF$ with $0 < |\Disc(x)| < 1$.
We approximate $\mathcal{F}^1$
by choosing a smooth (Schwartz class) function $\phi: V(\R) \rightarrow [0, 1]$ compactly supported within $\calF^1$.

Since the discriminant is homogeneous of degree $d$, 
the weighting function $\phi(x X^{-1/d})$ is a smoothed undercount of those $x \in G(\Z) \backslash V(\Z)$
with $0 < |\Disc(x)| < X$. That is, 
the role of the (counting function of the) set $\mathcal{A}(X)$ described in the introduction is taken
by the expression
\beq\label{eqn:smoothed_AX}
\sum_{\substack{x \in V(\Z)}} \phi(x X^{-1/d}).
\eeq

Although it won't be necessary here, it is possible to approximate $G(\Z) \backslash V(\R)$ as closely as we wish in the sense that,
for any $\beta < 1$, we may additionally require the locus $\mathcal{F}_\beta$ of $x$ with $\phi(x) = 1$ to satisfy
\beq
\frac{ \Vol(\mathcal{F}_\beta) }{ \Vol(\mathcal{F}^1)} > \beta.
\eeq


A number of variations are possible; for example we could restrict $\mathcal{F}^1$ to those 
$x \in \mathcal{F}$ with a particular sign, or choose $\phi$
to be supported away from any algebraic subset of $V(\R)$ defined by the vanishing of one or more homogeneous
equations.

\item ({\itshape Fourier transform formulas.}) For a squarefree integer $q$ we let $\Psi_q$ be the characteristic function of those
$x \in V(\Z)$ with $q \mid \Disc(x)$; this function factors through the reduction map
$V(\Z) \rightarrow V(\Z/q\Z)$. Its Fourier transform
$\widehat{\Psi_q}\colon V^\ast(\Z/q\Z) \rightarrow \C$
is defined by the usual formula
\beq\label{eqn:ft_intro}
\widehat{\Psi_q}(x) = q^{-r} \sum_{x' \in V(\Z/q \Z)} \Psi_q(x') \exp\left(\frac{2 \pi i [x', x]}{q}\right).
\eeq

The Fourier transform $\widehat{\Psi_q}$ is easily seen
to be multiplicative in $q$, and for $q = p \neq 2$ we
proved the following explicit formula
in \cite{TT_orbital}:
\begin{equation}\label{eq:Psi_fourier_ternary}
\widehat{\Psi_p}(x)=
\begin{cases}
p^{-1} + 2p^{-2} - p^{-3} - 2 p^{-4} - p^{-5} + 2p^{-6} + p^{-7} - p^{-8}
	& x \in \co_0,\\
p^{-3} - p^{-4} - 2p^{-5} + 2p^{-6} + p^{-7} - p^{-8}
	& x \in \odr,\\
2p^{-4} - 5p^{-5} + 3p^{-6} + p^{-7} - p^{-8}
	& x \in \ods,\\
p^{-4} - 3p^{-5} + 2p^{-6} + p^{-7} - p^{-8}
	& x \in \ocs,\\
- p^{-5} + p^{-6} + p^{-7} - p^{-8}
	& x \in \odi, \odg, \ocg, \ots, \oti,\\
-p^{-6} + 2p^{-7} - p^{-8}
	& x \in \cO_{1^2 1^2}, \\
p^{-6} - p^{-8}
	& x \in \cO_{2^2}, \\
p^{-7} - p^{-8}
	& x \in \co_{1^4}, \co_{1^3 1}, \co_{1^2 11}, \co_{1^2 2}, \\
-p^{-8}
	& x \in \cO_{1111}, \cO_{112}, \cO_{22}, \cO_{13}, \cO_{4}.
\end{cases}
\end{equation}

When $p \neq 2$ there are $20$ orbits for the action of
$G(\F_p)$ on $V^\ast(\F_p)$, and each has a description that is essentially uniform in $p$; these are denoted by
the $\cO$ above, or by $\cO(p)$ when we indicate the prime $p$ explicitly.
We refer to \cite{TT_orbital} for descriptions of each of the $\cO$,
together with computations of their cardinalities.

The $L_1$ norm of $\widehat{\Psi_p}(x)$ is $O(p^4)$ -- better than square root cancellation! In particular the larger contributions come from
the more singular orbits, and {\itshape our methods are designed to exploit this structure.}

What is required in general is that $\Psi_p$ be any bounded function,
which factors through the reduction map $V(\Z) \rightarrow V(\Z/p^a\Z)$
for some $a \geq 1$, for which we can compute or bound the Fourier transform.
(Incorporating the trivial bound $|\widehat{\Psi_p}(x)| \ll 1$ yields results which are in some sense nontrivial, but our interest is in doing better.)

In \cite{TT_orbital}, explicit formulas like \eqref{eq:Psi_fourier_ternary} are computed for any function $\Psi_p$ for which $\Psi_p(gx) = \Psi_p(x)$ for all
$g \in G(\F_p)$ when $a=1$.

\item ({\itshape Orbits in geometric terms.}) 
For each orbit description $\cO$,
there exists an integer $i = i(\cO) \in [0, d]$ such that $\# \co(p) \asymp_{\co} p^i$ as $p$ ranges; we call this integer the
dimension of $\cO$.
We will show in Section \ref{subsec:2sym23-subschemes}
that there also exists a closed {\itshape subscheme}
$\sillyX \subseteq V$ defined over $\Z$
of the same dimension $i(\cO)$, 
such that $\co(p) \subseteq \sillyX(\F_p)$ for all but (possibly) finitely many primes
$p$. As we will see, this will allow lattice point counting methods which use algebraic geometry.

\begin{remark} The `schemes' in question are simply 
the vanishing loci of systems of polynomials defined over $\Z$, and the algebraic geometry to be invoked will be fairly elementary.
However, one can study related problems using very sophisticated algebro-geometric tools; see for example \cite{DG, FK}.
\end{remark}

\end{enumerate}

\section{Levels of distribution and the weighted sieve}\label{sec:explain_ld}

We begin by discussing this sieve machinery we will apply.
In some (but not complete) generality, a {\itshape level of distribution} describes the following. Suppose that $a(n) : \Z^+ \rightarrow [0, \infty)$
is a function for which we can prove, for each squarefree integer $q$ (including $q = 1$), an estimate of the shape
\beq\label{eq:sieve2}
\sum_{\substack{n < X \\ q \mid n}} a(n) = \omega(q) C Y + E(X, q)
\eeq
for some constant $C$, multiplicative function $\omega(q)$ satisfying $0 \leq \omega(q) < 1$ for all $q$, function $Y$ of $X$, 
and 
error term $E(X, q)$.
With the setup described in Section \ref{sec:properties} we will have $Y = X^{r/d}$ with
\beq\label{eq:construct_weights}
a(n) = \sum_{\substack{x \in V(\Z) \\ |\Disc(x)| = n}} \phi(x X^{-1/d}).
\eeq

We say that the function $a(n)$ has {\itshape level of distribution}
$\alpha$ if for any $\epsilon > 0$ we have
\beq\label{eq:sieve3}
\sum_{q < X^{\alpha}} |E(X, q)| \ll_{\epsilon} Y^{1 - \epsilon},
\eeq
where the sum is over squarefree integers $q$ only. 
Bounds of the shape \eqref{eq:sieve3} are required for essentially all sieve methods, and also in many other 
analytic number theory techniques.

For our formulation of the {\itshape weighted sieve} we will 
also demand a {\itshape one-sided linear sieve inequality}
\beq\label{eq:ls}
\prod_{w \leq p < z} (1 - \omega(p))^{-1} \leq 
K \left( \frac{ \log z}{ \log w} \right)
\eeq
for all $2 \leq w < z$ and some fixed constant $K \geq 1$; the product is over primes. A familiar computation (see, for example,
\cite[(5.34)-(5.37)]{ODC}) shows that \eqref{eq:ls} holds if we assume for all prime $p$ that $w(p) < 1$ and that
\beq\label{eq:ls2}
\bigg|w(p) - \frac{1}{p}\bigg| < \frac{C}{p^2}
\eeq
for a fixed constant $C$, on which the constant
$K$ of \eqref{eq:ls} depends. (Conditions such as \eqref{eq:ls} and \eqref{eq:ls2} are often required in sieve methods,
and
may appear in a variety of guises.)

The {\itshape weighted sieve}, developed principally by Richert \cite{richert} and Greaves \cite{greaves}, and described here in the
formulation of Friedlander and Iwaniec \cite[Theorem 25.1]{ODC}, detects almost prime values of $n$ in the sequence $a(n)$.

\begin{theorem}[The weighted sieve \cite{richert, greaves}]\label{thm:ls}
Assume \eqref{eq:sieve2}, \eqref{eq:sieve3}, and \eqref{eq:ls}, and let $t \geq \frac{1}{\alpha} + \frac{\log 4}{\log 3} - 1$ be a positive integer.
Then we have
\beq\label{eq:ls_result}
\sum_{\substack{n \leq X \\ p \mid n \Rightarrow p > X^{\alpha/4} \\ \nu(n) \leq t}} a(n) \gg \frac{Y}{\log X},
\eeq
where $\nu(n)$ denotes the number of prime divisors of $n$.
\end{theorem}
This is one of many sieve methods which establish various consequences from 
hypotheses of the form \eqref{eq:sieve2}-\eqref{eq:ls}. We refer to 
\cite{ODC} for a nice overview of many different sieve methods and their applications, and 
to \cite{BBP, BF, B_quartic, B_quintic, B_geosieve, 
BCT, BST, FK, ST5, TT_rc} for applications concerning
prehomogeneous vector spaces. The papers \cite{BBP, B_quartic, B_quintic, BST, TT_rc} sieve rings for maximality, where the
analogue of $\omega(q)$ is roughly $1/q^2$; conversely, \cite{BCT, ST5} illustrate sieves where
$\omega(q)$ is not a decreasing function of $q$.

\smallskip
{\itshape Technical notes.}
Theorem \ref{thm:ls} may be deduced from the precise statement of Theorem 25.1 of \cite{ODC} as follows.
We take $N = 1$ in (25.7), corresponding to \eqref{eq:sieve2}. We choose $u = 1$
so that $\delta(u) = \frac{\log 4}{\log 3}$ (this is the limit as $u \rightarrow 1$ of the expression in (25.17)), and we
have $V(X) \gg_{\alpha} \log X$ by \eqref{eq:ls}.


As mentioned in \cite{ODC}, Greaves proved \cite[Chapter 5]{greaves} a related result with $\frac{\log 4}{\log 3} = 1.261\dots$ replaced with 
$1.124\dots$ Since we will eventually obtain $\alpha = \frac{7}{48}$ for the representation $\Z^2\otimes \Sym_2(\Z^3)$, this would yield
quartic field discriminants with only $7$ prime factors in Theorem \ref{thm:ap4}. But since our main goal is to showcase our sieve method,
we have chosen to apply a form of the weighted sieve
that is easier to extract from the literature. Note that the lower bound $p > X^{\alpha / 4}$ will be important in Section \ref{sec:ld_implies}.

\section{Smoothing and the Poisson summation formula}\label{sec:smoothing}

{\itshape Assumptions.} Until Section \ref{sec:reform}, the analysis in this section is quite general (and very standard). $V(\Z)$ will
denote a complete lattice in a vector space $V(\R)$ of dimension $r$. ($V$ itself will denote an $r$-dimensional affine space over $\Z$.)
In what follows $d$ will be the (homogeneous) degree of the discriminant polynomial, but in this section (where such a polynomial need not be defined) 
$d$ may be any positive real constant.
$X$ will be 
a positive real number; and for each squarefree
integer $q$, $\Psi_q : V(\Z) \rightarrow \C$ is any function which factors through the reduction map $V(\Z) \rightarrow V(\Z/q\Z)$.
All sums over $q$ will implicitly be over {\itshape squarefree} positive integers $q$ only.
We assume for simplicity that $|\Psi_q(x)| \leq 1$ for all $q$ and $x$. Finally, $\phi$ will denote any fixed smooth, Schwartz class function,
on which all implied constants below are allowed to depend.

\medskip

The aim of this section is to estimate the values of the sum
\begin{equation}\label{eq:intro_sum}
\sum_{x \in V(\Z)} \Psi_\ldmod(x) \phi(xX^{-1/d}),
\end{equation}
and in particular to prove upper bounds for the error terms, summed over $q$.
In the general setting of Section \ref{sec:properties} this is a smoothed undercount of
those $x \in G(\Z) \backslash V(\Z)$ with $0 < \pm \Disc(x) < X$ satisfying the 
congruence conditions implied by the function $\Psi_q$. 
In the more specific setting
of the proof of Theorem \ref{thm:ap4}, $\Psi_q$ is the characteristic function of those
$x$ with $q \mid \Disc(x)$, so that \eqref{eq:intro_sum} counts discriminants divisible by $q$.

By Poisson summation and a standard unfolding argument, we may check that
\begin{align}\label{eq:poisson_error}
\sum_{x \in V(\Z)} \Psi_\ldmod(x) \phi(xX^{-1/d}) & = X^{r/d} \sum_{x \in V^*(\Z)} \widehat{\Psi_\ldmod}(x) \widehat{\phi}\left(\frac{x X^{1/d}}{\ldmod} \right)\\
& = \widehat{\Psi_q}(0) \widehat{\phi}(0) X^{r/d} + E(X, \Psi_\ldmod, \phi),
\end{align}
where the error term $E(X, \Psi_\ldmod, \phi)$ is defined by
\beq\label{eq:def_error}
E(X, \Psi_\ldmod, \phi) := X^{r/d} \sum_{0 \neq x \in V^\ast(\Z)} \widehat{\Psi_\ldmod}(x) \widehat{\phi}\left(\frac{x X^{1/d}}{\ldmod} \right),
\eeq
and $\widehat{\phi}$ satisfies the rapid decay property 
\beq\label{eq:rapid_decay}
|\widehat{\phi}(y)| \ll_{A} (1 + |y|)^{-A}
\eeq
for every $A > 0$ and every
$y \in V^\ast(\R)$.
(Here $|y|^2 := y_1^2 + \cdots + y_{d}^2$.) 

\medskip
With an eye to \eqref{eq:sieve3}, we desire the following conclusion:
\begin{conclusion}[Level of distribution $\alpha$]\label{prop:gen_ld}
We have, for a parameters $\alpha > 0$ to be determined,
that the following inequality holds for
some $c < r/d$:
\beq\label{eq:gen_ld}
\sum_{\ldmod < X^{\alpha}} |E(X, \Psi_\ldmod, \phi)| \ll X^c.
\eeq
\end{conclusion}

\medskip

We will now prove that this conclusion is implied by the
more combinatorial statements of \eqref{eq:ld_simp} in
Proposition \ref{prop:gen_ld_simp} or \eqref{eq:ld_simp2} in Proposition \ref{prop:gen_ld_simp2}.

\smallskip 

For a parameter $Z > 0$, denote by $E_{\leq Z}(X, \Psi_\ldmod, \phi)$ the contribution to $E(X, \Psi_\ldmod, \phi)$ from those $x$ whose coordinates
are all bounded by $Z$, and write $E_{> Z}(X, \Psi_\ldmod, \phi)$ for the remaining contribution.

\begin{lemma}
For any $Z > 0$ and $A > d$ we have
\beq\label{eq:error_bound1}
E_{> Z}(X, \Psi_\ldmod, \phi)
\ll_{A}
X^{r/d} \left( \frac{\ldmod}{X^{1/d}} \right)^A Z^{-A + d},
\eeq
and if $Z := \ldmod X^{-1/d + \eta}$ for a fixed constant $\eta > 0$ and $\ldmod < X$ we have, for any $B > 0$,
\beq\label{eq:error_bound2}
E_{> Z}(X, \Psi_\ldmod, \phi)
\ll_{B, \eta} X^{-B}.
\eeq
\end{lemma}
\begin{proof}
By \eqref{eq:rapid_decay}, we have
\beq
E_{> Z}(X, \Psi_\ldmod, \phi)
\ll_{A} X^{r/d} \sum_{\substack{x \\ \exists i \ |x_i| > Z}}
\left( 1 + \frac{|x| X^{1/d}}{\ldmod} \right)^{-A}
\leq
X^{r/d} \left( \frac{\ldmod}{X^{1/d}} \right)^A  \sum_{\substack{x \\ |x| > Z}}
|x|^{-A}.
\eeq
There are $\ll R^{d}$ elements $x$ with $|x| \in [R, 2R]$ for any $R > 0$.
Therefore, assuming 
that $A > d$ this sum
is
\beq
\ll 
X^{r/d} \left( \frac{\ldmod}{X^{1/d}} \right)^A  \sum_{j = 0}^{\infty} (2^j Z)^{-A + d}
\ll
X^{r/d} \left( \frac{\ldmod}{X^{1/d}} \right)^A Z^{-A + d},
\eeq
proving \eqref{eq:error_bound1}. With $Z := \ldmod X^{-1/d + \eta}$
this simplifies to 
$X^{\frac{r}{d} - 1 + (d - A) \eta} \ldmod^{d} \leq X^{\frac{r}{d} - 1 + d + (d - A) \eta}$, and the result follows by choosing
$A = \frac{B + d + \frac{r}{d} - 1}{\eta} + d.$
\end{proof}

In what follows we will choose $Z = Z(\ldmod) := \ldmod X^{-1/d + \eta}$ for a fixed small $\eta > 0$ so as to guarantee 
\eqref{eq:error_bound2}, so that we have $E(X, \Psi_\ldmod, \phi) = O_{B, \eta, \phi}(X^{-B}) + E_{\leq Z}(X, \Psi_\ldmod, \phi)$,
with
\beq\label{eq:poisson_simp}
|E_{\leq Z}(X, \Psi_\ldmod, \phi)| \leq 
X^{r/d} \widehat{\phi}(0) \sum_{\substack{0 \neq x \in V^\ast(\Z) \\ |x_i| \leq Z \ \forall i }} 
|\widehat{\Psi_\ldmod}(x)|,
\eeq
with $\widehat{\phi}(0)$ being a convenient upper bound for $|\widehat{\phi}(t)|$.
We remark that if $\ldmod < X^{1/d - \eta}$ then the sum in \eqref{eq:poisson_simp} is 
empty and $E(X, \Psi_\ldmod, \phi) \ll_B X^{-B}$; i.e., the error is essentially zero. In general, we conclude the following:

\begin{proposition}[Level of distribution $\alpha$, simplified version]\label{prop:gen_ld_simp}
Conclusion \ref{prop:gen_ld} follows if we have, for the same
$\alpha > 0$, some $c < {r/d}$ and $\eta > 0$, every $N < X^{\alpha}$, 
and with $Z := 2 N X^{\eta - 1/d}$, that
\beq\label{eq:ld_simp}
X^{r/d}
\sum_{\ldmod \in [N, 2N]} 
\sum_{\substack{0 \neq x \in V^\ast(\Z) \\ |x_i| \leq Z \ \forall i}}
|\widehat{\Psi_\ldmod}(x)|
\ll X^c.
\eeq
\end{proposition}
\medskip
\begin{proof}
We divide the sum in \eqref{eq:gen_ld}
into $\ll \log X$ dyadic intervals $[N, 2N]$ and apply \eqref{eq:poisson_simp} to each $E(X, \Psi_\ldmod, \phi)$,
for each $q$ expanding the condition $|x_i| \leq Z(q)$ to $|x_i| \leq Z(N) = 2 N X^{-1/d + \eta}$. The term
$\widehat{\phi}(0)$ may be subsumed (for fixed $\phi$) into constants implied by the notation $\ll$ and $O(-)$,
and Conclusion \ref{prop:gen_ld} follows (with any $c$ strictly larger than that in \eqref{eq:ld_simp}, so as to incorporate
a contributions of $O(\log X)$ from the number of intervals).
\end{proof}

This statement may initially look more complicated than Conclusion \ref{prop:gen_ld}, but it is simpler in that it 
lends 
itself naturally to geometric proofs. Moreover the sums over $\ldmod$ and $x$ are independent
and can be interchanged.
\medskip

{\itshape The $L_1$ norm heuristic.} In the introduction, we said that
`$L_1$ norm bounds for Fourier transforms over finite fields should lead to level of distribution statements for arithmetic objects.'
Such a heuristic arises from \eqref{eq:ld_simp} by assuming that $|\widehat{\Psi_q}(x)|$ has the same average value
in the box defined by $|x_i| \leq Z$ as it does in all of $V^*(\Z/q\Z)$. Such a statement cannot be proved in general,
and indeed it is not always true: for example, in $\Z^2 \otimes \Sym_2 \Z^3$ there are disproportionately
many 
doubled forms $x = (x_1, x_1) \in V^*(\Z) \cap [-Z, Z]^{12}$ near the origin. That said, this heuristic is the motivation behind
our geometric sieve method, and it also provides a target for what we may hope to prove.

\subsection{Reformulation in terms of $V(\Z)$}\label{sec:reform}
In practice it will be convenient to describe the Fourier transforms $\widehat{\Psi_q}(x)$ in terms
of $V(\Z)$ instead of $V^*(\Z)$. To do this, assume we have a linear map $\rho : V^* \rightarrow V$,
defined by equations over $\Z$, satisfying the following properties\footnote{Formally we may define $\rho$ as a
morphism of schemes over $\Z$ (which is an isomorphism over $\Z[1/m]$); what we need is that
$\rho$ defines maps $V^*(\Z) \rightarrow V(\Z)$, $V^*(R) \rightarrow V(R)$ for each ring $R$ containing $\Z$,
and $V^*(\Z/q\Z) \rightarrow V(\Z/q\Z)$ for each quotient $\Z/q\Z$ of $\Z$, 
all defined
by the same equations and hence compatible with the appropriate ring homomorphisms.} for some integer $m$:
(1) We have
$m V(\Z) \subseteq \rho(V^*(\Z)) \subseteq V(\Z)$; (2) $\rho$ defines an isomorphism
$V^*(\Z/q\Z) \rightarrow V(\Z/q\Z)$ for all integers $q$ coprime to $m$; (3) for each $x \in V^*(\Z)$, the coefficients
of $\rho(x) \in V(\Z)$ are bounded above by $m$ times those of $x$. Note that (2) is implied by (1), since
$m V(\Z/q\Z) \subseteq \rho(V^*(\Z/q\Z)) \subseteq V(\Z/q\Z)$ for each $q$.

We then define $\widehat{\Psi_q}$ on $V(\Z/q\Z)$ by writing
$\widehat{\Psi_q}(x) = \widehat{\Psi_q}(\rho^{-1}(x))$, and we
lift this definition of $\widehat{\Psi_q}$ to all of $V(\Z)$.
Finally, by abuse of notation we write
$\widehat{\Psi_q}(x) = \widehat{\Psi_{\frac{q}{(q, m)}}}(x)$
for an arbitrary squarefree $q$, 
so that we have defined $\widehat{\Psi_q}(x)$
for all squarefree $q$ and all $x \in V(\Z)$.

The following is then immediate:

\begin{proposition}[Level of distribution $\alpha$, simplified version in terms of $V(\Z)$]\label{prop:gen_ld_simp2}
Given the constructions above, Conclusion \ref{prop:gen_ld} 
follows if we have, for the same
$\alpha > 0$, some $c < {r/d}$ and $\eta > 0$, every $N < X^{\alpha}$, 
and with $Z := N X^{\eta - 1/d}$, that
\beq\label{eq:ld_simp2}
X^{r/d}
\sum_{\ldmod \in [N, 2N]} 
\sum_{\substack{0 \neq x \in V(\Z) \\ |x_i| \leq Z \ \forall i}}
|\widehat{\Psi_\ldmod}(x)|
\ll X^c.
\eeq
\end{proposition}

In fact this conclusion is immediate only with $Z = 2m N X^{\eta - 1/d}$, but we observe that we may divide all
our previous choices of $Z$ by $2m$, with identical results holding at each step; alternatively we may take $\eta$ larger than 
that of Proposition \ref{prop:gen_ld_simp}. The implied constant in \eqref{eq:ld_simp2} is independent of $N$ and $X$
but may depend on the other variables.

\medskip

For each of the two specific representations
$(G,V)$ we treat in this paper,
as well as many other cases of interest,
such a $\rho$ is naturally induced by a non-degenerate symmetric bilinear form
$[-,-]$ on $V$, defined over $\Z[1/m]$, for which 
$[gx, g^{\iota} y] = [x, y]$ identically for an involution $\iota$ of $G$. Whenever $(q, m) = 1$, this implies that
$\rho : V^*(\Z/q\Z) \rightarrow V(\Z/q\Z)$ defines an isomorphism of $G(\Z/q\Z)$-modules.
These facts are important to our evaluation of the Fourier transforms $\widehat{\Psi_q}$ 
in \cite{TT_orbital}, which we describe as functions on $V(\F_q)$ rather than on $V^*(\F_q)$.

We refer to \cite{TT_orbital}, especially Sections 2 and A, for further details and explicit constructions.
For example, let $V$ be the space of binary cubic forms
with $G=\gl_2$.
Then $V(\Z)$ is the lattice of all integral binary cubic forms and
$m=3$.
The bilinear form on $V$ is defined by
\[
[x,x']=aa'+\frac13bb'+\frac13cc'+dd'
\]
for
$x=au^3+bu^2v+cuv^2+dv^3$
and
$x'=a'u^3+b'u^2v+c'uv^2+d'v^3$; the involution $\iota$ is defined by $g \mapsto g^{-T}$;
and
$\rho$ is the inverse of the map $V\ni x\mapsto [\cdot,x]\in V^\ast$, which is an isomorphism
over $\Z[1/3]$. Since $V^\ast(\Z)=\{\phi\in V^\ast(\Q)\mid \phi(V(\Z))\subset\Z\}$,
under the identification $V(\Q)=V^\ast(\Q)$,
$V^\ast(\Z)$ is the lattice of integral binary cubic forms
whose two middle coefficients are multiples of $3$,
and thus $V^\ast(\Z) \subset V(\Z)$.

For the space $V$ of pairs of ternary quadratic forms, 
$V(\Z)$ is the lattice of all pairs of integral quadratic forms,
$V^\ast(\Z)\subset V(\Z)$
is the lattice of pairs of integral quadratic forms
whose off-diagonal coefficients are multiples of $2$, and $m = 2$.

\section{Proof of Theorem \ref{thm:ap3}}\label{sec:ap3}

We now prove Theorem \ref{thm:ap3} by obtaining a level
of distribution of $\frac{1}{2} - \epsilon$ for a smoothed subset of integral orbits of binary cubic forms (where the level of distribution is again taken
with respect to the
property of the discriminant being divisible by $q$). Although we could appeal to the geometric sieve method of Section \ref{sec:ld}, we instead
give an easier proof whose idea is roughly equivalent to a special case of this method.

\medskip
In this section $V(\Z) := \Sym_3 \Z^2$ is the lattice of integral binary cubic forms, $r = d = 4$, $Z := NX^{-1/4 + \eta}$, $\Psi_q$ is the characteristic function of $x \in V(\Z)$ with
$q \mid \Disc(x)$, and we argue that we can prove the bound \eqref{eq:ld_simp2} for any $\alpha < \frac{1}{2}$, i.e. that
\begin{equation}\label{ld:cubic}
\sum_{\ldmod \in [N, 2N]} 
\sum_{\substack{0 \neq x \in V(\Z) \\ |x_i| \leq Z \ \forall i}}
|\widehat{\Psi_\ldmod}(x)|
\ll X^{-1 + c}
\end{equation}
for each $N < X^\alpha$ for some $c = c(\alpha) < 1$.
The Fourier transform $\widehat{\Psi_q}$ is multiplicative in $q$, and satisfies
\begin{equation}\label{eq:Psi_fourier_cubic}
\widehat{\Psi_p}(x)=
\begin{cases}
p^{-1} + p^{-2} - p^{-3}
	& \text{if } x \in pV(\Z),\\
p^{-2} - p^{-3}
	& \text{if } x \not \in pV(\Z) \text{ but } p \mid \Disc(x),\\
	- p^{-3} & \text{if } p \nmid \Disc(x).
\end{cases}
\end{equation}

For each positive divisor $q_0$ of $q$ and $x \in q_0 V(\Z)$, we have
\[
\widehat{\Psi_q}(x) = \widehat{\Psi_{q_0}}(x) \cdot \widehat{\Psi_{q / q_0}}(x)
= \widehat{\Psi_{q_0}}(x) \cdot \widehat{\Psi_{q / q_0}}(x / q_0).
\]
Therefore
the summation of \eqref{ld:cubic} is equal to
\begin{equation}\label{ld:cubic2}
\sum_{q_0 \leq Z} \Big( \prod_{p \mid q_0} (p^{-1} + p^{-2} - p^{-3}) \Big)
\sum_{\substack{\frac{N}{q_0} \leq q_1 \leq \frac{2N}{q_0} \\ (q_0, q_1) = 1}} 
\sump_{\substack{x \\ |x_i| \leq \frac{Z}{q_0} \ \forall i}}
|\widehat{\Psi_{q_1}}(x)|,
\end{equation}
where the inner sum is over those $x \in V(\Z)$ which are not in 
$p V(\Z)$ for any prime divisor $p$ of $q_0 q_1$.

We split the sum of \eqref{ld:cubic2} into two pieces: a sum over those $x$ for
which $\Disc(x) = 0$, and a sum over those $x$ for which $\Disc(x) \neq 0$.

\medskip
{\itshape Those $x$ with $\Disc(x) = 0$.} The number of such $x$ with all coordinates bounded by $Y$,
is $O(Y^2)$ for any $Y$. To see this, note that any such $x$ can be written
as $(ax + by)^2(cx + dy)$ for some $a, b, c, d \in \Z$. The number of possibilities with
$a = 0$ is $O(Y^2)$, as this forces the $x^3$ and $x^2 y$ coefficients to both be zero.
Similarly there are $O(Y^2)$ possibilities with $b = 0$. We are therefore left with the number
of integer quadruples $(a, b, c, d)$ with $a b \neq 0$, $|a^2 c| \leq Y$, and $|b^2 d| \leq Y$,
which is the square of the number of integer pairs $(a, c)$ with $a \neq 0$, $|a^2 c| \leq Y$.
This latter quantity is easily seen to be $O(Y)$, as needed.

The inner sum is therefore over $O(Z/q_0)^2$ elements, and 
for each $x$ we have $|\widehat{\Psi_{q_1}}(x)| \leq {q_1}^{-2}$. Therefore,
this portion of the sum in \eqref{ld:cubic2} is
\begin{equation}\label{ld:cubic2a}
\ll_{\epsilon} 
\sum_{q_0 \leq Z} q_0^{-1 + \epsilon} \cdot
\frac{N}{q_0} \cdot \left( \frac{Z}{q_0} \right)^2 \cdot \left( \frac{N}{q_0} \right)^{-2}
\ll_{\epsilon} N^{-1} Z^2 \sum_{q_0 \leq Z} q_0^{-2 + \epsilon}
\ll X^{2\eta} N X^{-1/2},
\end{equation}
which satisfies the bound \eqref{ld:cubic}.

\medskip
{\itshape Those $x$ with $\Disc(x) \neq 0$.}
The contribution of these is bounded above by
\begin{align*}\label{ld:cubic2b}
& \sum_{q_0 \leq Z} \Big( \prod_{p \mid q_0} (p^{-1} + p^{-2} - p^{-3}) \Big)
\sum_{\substack{\frac{N}{q_0} \leq q_1 \leq \frac{2N}{q_0} \\ (q_0, q_1) = 1}} q_1^{-3}
\sump_{\substack{x \\ |x_i| \leq \frac{Z}{q_0} \ \forall i}}
\gcd(\Disc(x), q_1)
\\
\leq
& \sum_{q_0 \leq Z} \Big( \prod_{p \mid q_0} (p^{-1} + p^{-2} - p^{-3}) \Big)
\left( \frac{N}{q_0} \right)^{-3}
\sum_{\substack{x \\ |x_i| \leq Z/q_0 \ \forall i \\ \Disc(x) \neq 0}}
\sum_{\substack{\frac{N}{q_0} \leq q_1 \leq \frac{2N}{q_0}}} 
\gcd(\Disc(x), q_1).
\end{align*}
Now, in general, whenever $m \neq 0$ we have
\[
\sum_{n \in [N, 2N]} \gcd(m, n)
\leq
\sum_{\substack{ f \mid m \\ f \leq 2N}} f 
\sum_{\substack{ n \in [N, 2N] \\ f \mid n}} 1
\leq
\sum_{\substack{ f \mid m \\ f \leq 2N}} f 
\left( \frac{N}{f} + 1 \right)
\ll N m^{\epsilon}.
\]
Therefore, using that the discriminant of any $x$ in the sum is $\ll N^4$, we see that the previous quantity
is
\begin{align*}\label{ld:cubic2c}
\ll_{\epsilon}
& \sum_{q_0 \leq Z} \Big( \prod_{p \mid q_0} (p^{-1} + p^{-2} - p^{-3}) \Big)
\left( \frac{N}{q_0} \right)^{-3}
\cdot
\left( \frac{Z}{q_0} \right)^4
\frac{N}{q_0} \cdot N^{\epsilon}
\\
\ll_{\epsilon} &
N^{\epsilon} Z^4 N^{-2}
\sum_{q_0 \leq Z} q_0^{-3 + \epsilon} 
\\
\ll_{\epsilon} &
N^{\epsilon} N^2 X^{-1 + 4 \eta},
\end{align*}
again satisfying \eqref{ld:cubic}.
\medskip

Applying the weighted sieve of Theorem \ref{thm:ls}, and following the beginning of the proof of Proposition \ref{prop:ld_gives_main},
we obtain $\gg \frac{X}{\log X}$ elements $x \in V(\Z)$ whose discriminants
have
at most three prime factors. Of these, at most $O_\epsilon(X^{3/4 + \epsilon})$ can be reducible. (For a simple proof see \cite[Lemma 21]{BST}; only 
the second paragraph of the proof there is relevant, as we are counting points in a box of side length $\ll X^{1/4}$.) As the weighted sieve produces $x$ with
each prime factor
$> X^{\alpha/4}$, the number of $x \in V(\Z)$ with any repeated prime factor is
$\ll \sum_{p > X^{\alpha/4}} \frac{X}{p^2} \ll X^{1 - \alpha/4}$ by \cite[Lemma 3.4]{BBP}. Accordingly we produce $\gg \frac{X}{\log X}$ irreducible elements $x \in V(\Z)$ with
squarefree (and hence fundamental) discriminants, which must therefore correspond to (distinct) maximal cubic orders and hence to  cubic fields.

\begin{remark} 
In place of our estimate of $O(Y^2)$ for reducible elements $x$ in boxes of side length $O(Y)$,
the method of Section \ref{sec:ld} would implicitly incorporate 
a bound of $O(Y^3)$, as $3$ is the dimension of the variety $\Disc(x) = 0$. This proof illustrates that
counting elements more directly may yield improvements in the end results.
\end{remark}

\section{Closed subschemes containing singular orbits}
\label{subsec:2sym23-subschemes}

Let $V$ be the space of pairs of ternary quadratic forms, together with its
action of $G = \GL_2 \times \GL_3$. Recall from \cite[Proposition 21]{TT_orbital}
that, for each $p \neq 2$ there are $20$ orbits for the action of $G(\F_p)$ on $V(\F_p)$.
We gave twenty `orbit descriptions' $\co$ which were essentially uniform in $p$, and for each $p$
we write $\co(p)$ for the associated orbit over $\F_p$. 

\begin{proposition}\label{prop:same_dim}
For each of the orbit descriptions $\co$ described above there exists a closed subscheme
$\sillyX \subset V$, defined over $\Z$, such that $\co(p) \subseteq \sillyX(\F_p)$ for each prime $p \neq 2$, and of the same
`dimension' as 
$\co$ in the sense that $\# \co(p) \asymp \# \sillyX(\F_p)$.
\end{proposition}

We will prove this statement, with `$p \neq 2$' replaced with `$p \not \in S$ for some finite set $S$',
for any finite dimensional $(G, V)$ satisfying the following two properties:
\begin{itemize}
\item
There exist finitely many elements $x_\sigma\in V(\Z)$
such that
for any algebraically closed field $K$ with ${\rm char}(K)\not \in S$,
the images of the $x_\sigma$ in $V(K)$ via the
canonical map $V(\Z)\rightarrow V(K)$
form a set of complete representatives
for $G(K)\backslash V(K)$.
\item
There exists a constant $c > 0$ such that for each $p \not \in S$ and $G(\F_p)$-orbit $\co \subseteq V(\F_p)$
we have $\# \co > c \# \widetilde{\co}$, with $\widetilde{\co} := G(\overline{\F_p}) \co \cap V(\F_p)$.
\end{itemize}
These properties hold for all of the $(G, V)$ studied in \cite{TT_orbital}, as we explain now in the case of 
pairs of ternary quadratic forms.
We group the $20$ orbit descriptions of \cite[Proposition 21]{TT_orbital} as follows:
\begin{multline*}
\{ \co_0 \},
\{ \co_{D1^2} \},
\{ \co_{D11}, \co_{D2} \},
\{ \co_{D\rm ns} \},
\{ \co_{C\rm s} \},
\{ \co_{C\rm ns} \},
\{ \co_{B11}, \co_{B2} \}, \\ 
\{ \co_{1^4} \},
\{ \co_{1^3 1} \},
 \{ \co_{1^2 1^2}, \co_{2^2} \},
 \{ \co_{1^2 11}, \co_{1^2 2} \},
 \{ \co_{1111},  \co_{112,}  \co_{22},
\co_{13}, \co_4 \}.
\end{multline*}

Within each of these twelve sets, the orbital representative of the {\itshape first-listed} $\co$ is described
in \cite[Proposition 21]{TT_orbital} as the reduction $\pmod p$ of a fixed element $x_\sigma \in V(\Z)$, and
when $K$ is algebraically closed with $\chr(K) \neq 2$, the proof
in \cite[Section 7.1]{TT_orbital} establishes that the images of these $x_\sigma$ in $V(K)$ are a set of representatives for 
$G(K) \backslash V(K)$. Moreover, for each $p \not \in S$ and $G(\F_p)$-orbit $\co \subseteq V(\F_p)$, the associated
$\widetilde{\co}$ is the union of the $\co$ in the grouping described above, and the second property above may
be deduced from the point counts in \cite[Proposition 21]{TT_orbital}.

To conclude Proposition \ref{prop:same_dim} from these two properties, write
$Y_\sigma := G(\overline{\Q}) x_\sigma \subseteq V(\overline{\Q})$ for each $x_\sigma$. By 
\cite[Propositions I.1.8 and II.6.7]{borel} we may write each $Y_\sigma$ in the form
$Y_\sigma = \sillyX_\sigma - \cup_j W_{\sigma, j}$ where the $\sillyX_\sigma$ and $W_{\sigma, j}$ are (finitely many)
closed varieties, defined over $\Q$, and with $\dim(W_{\sigma, j}) < \dim(\sillyX_\sigma)$ for all $\sigma$ and $j$.
(Each $\sillyX_\sigma$ is the closure of $Y_\sigma$, and the $W_{\sigma, j}$ are defined by the closures of other $G(\overline{\Q})$-orbits, of which there are finitely many,
and since each of the $x_\sigma$ is defined over $\Q$ their orbits are as well.)

We choose (arbitrary) integral models for the
$\sillyX_\sigma$ so as to regard them as closed subschemes of $V$.
For all but finitely many $p$,
these equations reduce $\pmod p$ and define varieties of the same dimension over $\F_p$, and we conclude
by Lang-Weil \cite{LW} that $\#\sillyX_\sigma(\F_p) \asymp \# Y_\sigma(\F_p) \asymp p^{\dim(\sillyX_\sigma)}.$ The second bullet point above then gives
the desired conclusion.

\section{A version of the geometric sieve}\label{sec:geom}

The {\itshape Ekedahl-Bhargava geometric sieve}, in the form of \cite[Theorem 3.3]{B_geosieve}, asserts the following.
Let $B$ be a compact region in $\R^r$, let $\sillyX$ be a closed subscheme of $\A_\Z^r$ of codimension
$a \geq 1$, and let $\lambda$ and $P$ be positive real numbers. Then, we have
\begin{equation}\label{eq:b_geo}
\# \{ x \in \lambda B \cap \Z^r \ | \ x \pmod p \in \sillyX(\F_p) 
\textnormal{ for some prime } p > P \}
= O_{B, \sillyX} \Bigg( \frac{\lambda^r}{P^{a - 1} \log P} + \lambda^{r - a + 1} \Bigg).
\end{equation}

We introduce a variation with two modifications. Firstly, we count each $x$ with multiplicity, given by
the number of pairs $(x, p)$ for which $x \pmod p \in \sillyX(\F_p)$ and $p \in [P, 2P]$.
Secondly, we introduce an `arithmetic progression' condition $x - x_0 \in m V(\Z)$, allowing for additional flexibility
in applications (as we will see in Section \ref{sec:ld}.) 

We refer also to \cite{BCT} where the same generalization is presented concurrently; the application there
replaces {\itshape primes} in \eqref{eq:b_geo} with squarefree integers, which amounts roughly to a simpler version of the argument in Section \ref{sec:ld} here.

\begin{theorem}\label{lem:geosieve}
Let $B$, $\sillyX$, $a$, and $\lambda$ be as in the statement of \eqref{eq:b_geo}, let $m$ be a positive integer, let $x_0 \in V(\Z)$,
and let $P > \lambda/m > 1$ be any real number.
Then, we have
\begin{equation}\label{eq:beo}
\# \{ (x, p) \ | \ x \in \lambda B \cap (x_0 + m\Z^r), \ \ \textnormal{$p$ is a prime in $[P, 2P]$}, \  p \nmid m, \ 
x \pmod p \in \sillyX(\F_p) \}
\ll_{B, \sillyX, \epsilon} \Big( \frac{\lambda}{m} \Big)^{r - a} P \lambda^{\epsilon}.
\end{equation}

\end{theorem}

\begin{proof} This closely follows the proof of Theorems 3.1 and 3.3 of Bhargava \cite{B_geosieve}. In \cite{B_geosieve} (with $m = 1$),
the quantity $\lambda$ appears only as an upper bound for the number of lattice points in $rB$ lying on a line defined by fixing all
but one of the coordinates. Therefore, with $m > 1$ we can replace $\lambda$ with $\frac{\lambda}{m}$ at each occurrence.

The analogue of \cite[Lemma 3.1]{B_geosieve}, proved identically, thus reads that
\begin{equation}\label{eq:vz}
\# \{ x \in \lambda B \cap (x_0 + m\Z^r) \cap \sillyX(\Z) \}
\ll_{B, \sillyX} \Big( \frac{\lambda}{m} \Big)^{r - a},
\end{equation}
and we obtain the bound of \eqref{eq:beo} for those $x \in \sillyX(\Z)$.

For those $x \not \in \sillyX(\Z)$, it suffices to prove that
\begin{equation}\label{eq:beo2}
\# \left\{ (x, p) \ \Bigg| 
\begin{array} {l}
x \in \lambda B \cap (x_0 + m\Z^r), \ x \not \in \sillyX(\Z), \\
p > \frac{\lambda}{m}, \  p \nmid m, \ 
x \pmod p \in \sillyX(\F_p) 
\end{array}
\right\}
\ll_{B, \sillyX, \epsilon} \Big( \frac{\lambda}{m} \Big)^{r - a + 1} \cdot \lambda^{\epsilon},
\end{equation}
the exact analogue of \cite[(17)]{B_geosieve}. 
This follows \cite{B_geosieve} exactly. The condition $p \nmid m$ is
needed in the last paragraph of \cite[Theorem 3.3]{B_geosieve}, to conclude that if $f_k(x)$ is a polynomial in one variable
with $f_k \not \equiv 0 \pmod p$, then it has $O_{\deg(f_k)}(1)$ roots $x$ in an interval of length $O(\lambda)$, and with
$x \equiv x_0 \pmod m$ for any fixed $x_0$. The factor of $\lambda^{\epsilon}$ arises in adapting the argument immediately after \cite[(17)]{B_geosieve}:
any nonzero $f_i(x)$ can have only $O_{f_i}(1)$ prime factors $p > \lambda$, but it may have $O_{f_i, \epsilon}(\lambda^{\epsilon})$
prime factors $p > \lambda/m$.
\end{proof}

\section{Application of the geometric sieve: Proof of Proposition \ref{prop:gen_ld_simp2}}\label{sec:ld}

In this section we prove the bound \eqref{eq:ld_simp2} for each $\alpha < \frac{7}{48}$ for the `quartic' $(G, V)$ of Section \ref{sec:properties}.

For each $p \neq 2$ and
$i \in \{0, 4, 7, 8, 10, 11, 12 \}$ we define sets 
$U_i(p)$,  
each of which is a union of $G(\F_p)$-orbits on $V(\F_p)$, as follows.

\[
\begin{array}{c | c | c c}
\textnormal{Label} & \textnormal{Consists of}
& \text{Dimension} \ i & \text{Fourier contribution} \ \text{fc}(i)
\\ \hline
U_0 & 
\co_0 & 0 & -1 \\
U_4&
\co_{D1^2}
& 4 & -3 \\
U_7
&
\co_{D11}, \co_{C\rm s}
& 7 & -4 \\
U_8
&\co_{T11}, \co_{T2}, \co_{D2}, \co_{D\rm ns}, \co_{C\rm ns} &
8 & -5 \\
U_{10}&
\co_{1^2 1^2}, \co_{2^2}, \co_{1^3 1}, \co_{1^4}
& 10 & -6 \\
U_{11}&
\co_{1^2 11}, \co_{1^2 2}
& 11 & -7 \\
U_{12}&
\text{nonsingular orbits}
& 12 & -8 \\
\end{array}
\]
In Section \ref{subsec:2sym23-subschemes} we proved that for each $i \in  \{0, 4, 7, 8, 10, 11, 12 \}$ 
there are subschemes $\sillyX_i$  of $\mathbb{A}^r_{\Z} = V$ of dimension $i$, defined over $\Z$,
for which $U_i(p) \subseteq \sillyX_i(\F_p)$ and $\# \sillyX_i(\F_p) \ll p^i$ for all $p \not \in S$.
The function $\fc(i)$ is chosen such that
$|\widehat{\Psi_p}(x)| \leq 2 p^{\fc(i)}$ for each $x \in U_i(p)$. 

For every squarefree $n \in [N, 2N]$ (with $N < X^{\alpha}$ for $\alpha$ to be determined)
we consider the contribution to \eqref{eq:ld_simp2} from 
every factorization 
\begin{equation}\label{eq:factor_n}
n = n_0 n_4 n_7 n_8 n_{10} n_{11} n_{12}
\end{equation}
and those
$x$ with $x \in U_i(p)$ for each $p \mid n_i$.
When $n$ is even we will assume as a bookkeeping device that
$n_{0}$ is as well, but we will never
demand any geometric condition on $x$ modulo $2$.

The contribution of each such $x$ is bounded above by
$2^{\omega(n) + 1} \prod_i n_i^{\fc(i)}$, where the $2^1$ factor reflects the trivial
bound $|\widehat{\Psi_2}(x)| \leq 1$,
and we write
$2^{\omega(n)} = O_{\epsilon}(X^{\epsilon})$, uniformly in $n$.

We consider the following choices of parameters:
\begin{itemize}

\item Squarefree and pairwise coprime integers $n_i$ for $i \in \{ 0, 4, 7, 8, 10, 11, 12 \}$, with $\prod_i n_i \in [N, 2N]$.

\item A parameter $j \in \{ 4, 7, 8, 10, 11, 12 \}$, and a factorization $n_j = n_j' p n_j''$, where $p$ is a prime. 

\item Writing $m := n_j' \prod_{i < j} n_i$, these choices are subject to the condition that 
$m \leq Z < mp$.
\end{itemize}

We claim that every factorization \eqref{eq:factor_n} corresponds to at least one choice of the above data,
with $n_j = n_j' p n_j''$. First of all, note that $n_0 \leq Z$ for each
nonzero $x \in [-Z, Z]^{r}$.
Thus, given any factorization \eqref{eq:factor_n}, we let $j \geq 4$ be the minimal index
with $\prod_{i \leq j} n_i > Z$, choose $n_j'$ to be the largest divisor of $n_j$ less than or equal to
$Z \prod_{i < j} n_i^{-1}$, and choose $p$ to be any prime divisor of $n_j / n_j'$.

\medskip

The conditions modulo $m$, namely that $x \in U_i(q)$ for each odd  prime $q \mid m$ with $i = i(q)$ determined by the factorization
above, are equivalent to demanding that $x$ lie in one of $O_{\epsilon}(X^{\epsilon} n_j'^{j} \prod_{i < j} n_i^i)$ 
residue classes $\pmod{m V(\Z)}$.
(Here
$X^{\epsilon}$ is a simple upper
bound for $C^{\omega(n)}$, the product of the implied
constants occurring in the point counts for the $U_i(p)$.)

We must have $x \in \sillyX_{j}(\F_p)$, and for each of the residue classes
$\pmod{m V(\Z)}$ determined above we 
use Bhargava's geometric sieve
(Theorem \ref{lem:geosieve}) to bound the number 
of pairs $(x, p)$ where $x \in V(\Z)$ lies in this residue class, 
has all coefficients bounded by $Z$, and lies in $\sillyX_j(\F_p)$, and where
$p \not \in S$ lies in a dyadic interval $[P, 2P]$. By the theorem, the number of such pairs is
$\ll Z^\epsilon (Z/m)^{j} P$.

(Any contribution of $(x, p)$ with $p$ in the exceptional set $S$ of Proposition \ref{prop:same_dim} trivially satisfies the same bound, as in this case 
$Z/m \ll_S 1$.)

The Fourier contribution of each $x$ being counted is
$\ll X^{\epsilon} \prod_i n_i^{\fc(i)}$, and for each choice of
$j$, $n_i$ ($i < j$), and $n_j'$, and for each fixed dyadic interval
$[P, 2P]$, we multiply: the number of residue classes modulo $mV(\Z)$;
the number of pairs $(x, p)$ in each; the Fourier contribution of each $x$ being counted;
and the $N^{1 + \epsilon}/mP$ choices of $n_j''$ and $n_i$ ($i > j$).
Recalling that $Z = NX^{-1/d + \eta}$, we conclude that the contribution to 
\eqref{eq:ld_simp2} from the choices previously determined is
\[
\ll_{\epsilon} X^{\epsilon} \cdot X \cdot
n_j'^{j} \prod_{i < j} n_i^i
\cdot 
\left( \frac{NX^{-1/d + \eta}}{m} \right)^j P \cdot
\prod_i n_i^{\fc(i)} \cdot \frac{N^{1 + \epsilon}}{mP}.
\]
Using the fact that $\fc(i)$ is a decreasing function of $i$, and summing over the $\ll X^{\epsilon}$ choices of dyadic interval
$[P, 2P]$, we see that this is
\[
\ll_{\epsilon} X^{\epsilon + r \eta} \cdot X^{1 - j/d} \cdot
\left(
\prod_{i < j} n_i^{i + \fc(i)}
\right)
\cdot 
(n_{j}')^{j + \fc(j)}
\cdot
\left( \frac{N}{m} \right)^{j + \fc(j) + 1}
\]
Now, since $i + \fc(i)$ is an increasing function of $i$ this is bounded above by
\begin{align*}
\ll_{\epsilon} & X^{\epsilon + r \eta} \cdot X^{1 - j/d} \cdot
m^{j + \fc(j)}
\cdot
\left( \frac{N}{m} \right)^{j + \fc(j) + 1}
\\
\ll_{\epsilon} & X^{\epsilon + r \eta} \cdot X^{1 - j/d} 
\cdot m^{-1} N^{j + \fc(j) + 1},
\end{align*}
and, now fixing only the parameter $j$, 
we sum over all $m \leq Z$ and (for each $m$)
the $\ll N^{\epsilon}$ choices of factorizations of $m$ to
obtain a total contribution 
\beq\label{eq:final_cont}
\ll_{\epsilon} X^{\epsilon + r \eta} \cdot X^{1 - j/d} 
\cdot N^{j + \fc(j) + 1}
\eeq
from all choices of \eqref{eq:factor_n} associated to this factor $j$. Up to an implied constant,
the total error is bounded above by the maximum of \eqref{eq:final_cont} over the six admissible values of $j$.
The quantity in \eqref{eq:final_cont} is:
\[
\begin{array}{c | c}
\textnormal{$j$} & \textnormal{Bound} (\times X^{\epsilon + r \eta})
\\ \hline
j  = 4 & X^{2/3} N^2\\
j  = 7 & X^{5/12} N^4\\
j  = 8 & X^{1/3} N^4\\
j  = 10 & X^{1/6} N^5\\
j  = 11 & X^{1/12} N^5\\
\ j  = 12 & N^5\\
\end{array}
\] 
The case $j = 7$ turns out to be
the bottleneck, and choosing $N = X^{\alpha}$ with any $\alpha < \frac{7}{48}$ we may choose $\eta$ and $\epsilon$ with $c := \frac{5}{12} + 4 \alpha + \epsilon + 12 \eta < 1$,
so that \eqref{eq:ld_simp2} holds with this value of $c$.

\section{Conclusion of the proof of Theorem \ref{thm:ap4}}\label{sec:ld_implies}
We give a slightly more general statement, which illustrates how improvements to Theorem \ref{thm:ap4} would automatically
follow from improvements in the level of distribution.

\begin{proposition}\label{prop:ld_gives_main}
Assume, for some integer $t \geq 1$, that Proposition \ref{prop:gen_ld_simp2} (and hence
Conclusion \ref{prop:gen_ld}) holds for some $c < 1$ and $\alpha > \Big(t + 1 - \frac{\log 4}{\log 3}\Big)^{-1}$.
Then there are $\gg_{t, \alpha, c} \frac{X}{\log X}$ $S_4$-quartic field discriminants $K$
with $|\Disc(K)| < X$, such that $\Disc(K)$ has at most $t$ prime factors.
\end{proposition}

Since we proved Proposition \ref{prop:gen_ld_simp2} with any $\alpha < \frac{7}{48}$, we thus obtain
Theorem \ref{thm:ap4} with any $t > \frac{48}{7} - 1 + \frac{\log 4}{\log 3} = 7.119\dots$, and in particular with
$t = 8$.

\medskip
\begin{proof} 
We apply the weighted sieve of Theorem \ref{thm:ls}, with $Y = X$ and
\begin{equation}\label{eqn:ws_sum}
a(n) := \sum_{\substack{x \in V(\Z) \\ |\Disc(x)| = n}} \phi(x X^{-1/12}).
\end{equation}
Each sum is finite because $\phi$ is compactly supported. We then have
\[
\sum_{\substack{n < X \\ \ldmod \mid n}} a(n) = \sum_{\substack{x \in V(\Z)}} \Psi_\ldmod(x) \phi(x X^{-1/12}),
\]
where $\Psi_\ldmod$ is the characteristic function of $x \in V(\Z)$ with $\ldmod \mid \Disc(x)$.
By \eqref{eq:poisson_error}-\eqref{eq:def_error} the sequence satisfies the sieve axiom 
\eqref{eq:sieve2}, and by assumption Proposition \ref{prop:gen_ld_simp2} and therefore Conclusion \ref{prop:gen_ld} and
\eqref{eq:sieve3} hold. The linearity conditions \eqref{eq:ls} and \eqref{eq:ls2} follow from the first line of \eqref{eq:Psi_fourier_ternary}.

Theorem \ref{thm:ls} therefore implies that the sum of $\phi(x X^{-1/12})$, over all $x$
whose discriminants have at most $t$ prime factors, is $\gg \frac{X}{\log X}$. By construction
the count of such $x$ satisfies the same lower bound, and 
these discriminants are all in $(-X, 0) \cup (0, X)$ and are $G(\Z)$-inequivalent in $V(\Z)$.

By Bhargava \cite[Theorem 1]{HCL3}, these are in bijection with pairs $(Q, R)$, where $Q$ is a quartic ring
and $R$ is a cubic resolvent ring of $R$, and in case $Q$ is maximal then it has
exactly one cubic resolvent \cite[Corollary 5]{HCL3}.
Moreover, if $x \in V(\Z)$ corresponds to $(Q, R)$, then
$\Disc(x) = \Disc(Q)$.
As described on \cite[p. 1037]{HCL3},
$Q$ is an order in an $S_4$- or $A_4$-field
if and only if the corresponding $x \in V(\Z)$ is absolutely irreducible.

The number of $x$ which are not absolutely irreducible is $\ll X^{11/12 + \epsilon}$ and hence negligible; this is proved in
\cite[Lemmas 12 and 13]{HCL3}. In our case these proofs simplify
because we may ignore the cusp: the compact support of $\phi$ ensures that we are only counting points in a box of side length $O(X^{1/12})$, and that the number of points
with $a_{11} = 0$ is $O(X^{11/12})$.

We must then bound the number of pairs $(Q, R)$ where $Q$ is a nonmaximal $S_4$- or $A_4$-quartic order.
By Theorem \ref{thm:ls} the discriminants of
$x \in V(\Z)$ being counted have all of their prime factors $> X^{\alpha / 4}$,
and in particular any nonmaximal $Q$ which survives the sieve must be nonmaximal at
some prime $p > X^{\alpha/4}$. By \cite[Proposition 23]{B_quartic}, the 
number of such $x$ is
\[
\ll \sum_{p > X^{\alpha/4}} X/p^2 \ll X^{1 - \alpha/4},
\]
negligible for any $\alpha > 0$. 

This leaves the maximal $Q$ whose discriminants are divisible by $p^2$ for some $p$ in the same range. These can be handled
by the geometric sieve, precisely as Bhargava did in \cite{B_geosieve}. We apply the geometric sieve in its original formulation
directly to \eqref{eqn:ws_sum}, in contrast
to Section \ref{sec:ld} where we applied our variation after an application of Poisson summation.

Any maximal $Q$ whose discriminant is divisible by $p^2$ must be, in the language of Bhargava \cite{B_geosieve}, a {\itshape strong multiple}
of $p$; and hence (as in Section \ref{subsec:2sym23-subschemes})
the corresponding $x$ must be in $\sillyX(\F_p)$ for a suitably defined subscheme $\sillyX \subseteq V$ of codimension $2$.
By \cite[Theorem 3.3]{B_geosieve}, 
the number of such $x$ which satisfy this criterion for any
$p > X^{\alpha/4}$ is again $\ll X^{1 - \alpha/4}$.

In conclusion, the contributions of everything other than maximal orders in $S_4$-quartic fields to our sieve result is negligible,
and hence we obtain $\gg \frac{X}{\log X}$ $S_4$-quartic fields with the stated properties.
\end{proof}

\section*{Acknowledgments}
We would like to thank Theresa Anderson, Manjul Bhargava, Alex Duncan, \'Etienne Fouvry, Yasuhiro Ishitsuka, Kentaro Mitsui, Arul Shankar, Ari Shnidman, Jack Thorne, Jerry (Xiaoheng) Wang,
Melanie Matchett Wood and Kota Yoshioka for helpful discussions, comments and feedback.
Duncan, in particular, explained to us the application of the lemma from \cite{borel} to
the proof of Proposition \ref{prop:same_dim}.

Taniguchi's work was partially supported by the JSPS, KAKENHI Grant Numbers JP24654005, JP25707002, and JP17H02835.
Thorne's work was partially supported by the National Science Foundation under Grant No. DMS-1201330 and by the National Security Agency under a Young Investigator Grant.

\bibliographystyle{plain}
\bibliography{expo-sums}

\end{document}